\documentclass[draft, 12pt]{amsart}

\usepackage{amsmath,amssymb,amsthm,verbatim}

\newtheorem{df}{Definition}
\newtheorem{lem}{Lemma}
\newtheorem{prop}{Proposition}

\newtheorem{cor}{Corollary}

\newtheorem{claim}{Claim}
\newtheorem{thm}{Theorem}

\begin{document}

\newcommand{\sinc}{\text{sinc}}

\title[Regular Cardinal Interpolators]{On the convergence of regular families of cardinal interpolators}

\author{Jeff Ledford}


\begin{abstract}
In this note a general way to develop a cardinal interpolant for $l^2$-data on the integer lattice $\mathbb{Z}^n$ is shown.  Further, a parameter is introduced which allows one to recover the original Paley-Weiner function from which the data came.

\smallskip
\noindent \textsc{Keywords.} \it{cardinal interpolation, multiquadric, multivariate interpolation, Paley-Wiener functions}
\end{abstract}

\maketitle

\section{Introduction}

In this article we prove a result similar to the ones found in \cite{baxter}, \cite{splines}, \cite{RS 1}, \cite{RS 2}, and \cite{RS}.  These results may be thought of as outgrowths of Schoenberg's work on splines (see for instance \cite{Schoenberg}).  The general set up is as follows.  Suppose that we have data $\{f(j)\}$ from some class, we wish to interpolate this data with an interpolant that depends on some parameter.  In each case what is investigated is what happens as the parameter approaches a limiting value.  The papers mentioned above all fix a particular type of what it called in this article a cardinal interpolator.  The goal of this paper is to unify the procedures in those and similar papers in order to provide conditions on an interpolator which allows one to recover Paley-Wiener functions from their samples on the integer lattice by allowing the parameter to approach a limiting value.  In doing so, more examples are uncovered.

This article is organized as follows.  In the next, section several preliminary definitions are laid out.  We define what is meant by cardinal interpolator in Section 3 and introduce the corresponding fundamental function.  Section 4 deals with properties of the fundamental function, while in Section 5 regular families are introduced and contains the proof of the main theorem.  The final section consists of three examples of regular families of cardinal interpolators.

Finally, our calculations will often involve a positive constant $C$, whose value depends on its occurrence but not on various parameters unless otherwise stated.

\section{Definitions and Basic Facts}
In this section, we collect several definitions and general facts that will be used in what follows.  Our methods require the use of the Fourier transform, for which we adopt the following convention.

\begin{df} If $f\in L^1(\mathbb{R}^n)$ we define its \emph{Fourier transform}, denoted $\hat{f}(\xi)$ or $\mathcal{F}(f)(\xi)$, to be:
\begin{equation}\label{FT definition}
\mathcal{F}(f)(\xi)=\hat{f}(\xi)=(2\pi)^{-n/2}\int_{\mathbb{R}^n}f(x)e^{-ix\cdot\xi}dx,
\end{equation}
where $x\cdot\xi=\sum_{j=1}^{n}x_j\xi_j$.
\end{df}

We will encounter examples which are not integrable in this case we adopt the convention of \cite{wendland}, particularly Definition 8.9 on page 103,  and include it here for convenience changing the name a little to avoid confusion.

\begin{df}  For $m\in\mathbb{Z}^+$, the set of all functions $\gamma\in\mathcal S$ that satisfy $\gamma(\xi)=O(\|\xi\|^m)$ for $\|\xi\|\to0$ will be denoted by $\mathcal{S}_{m}$.\\
Here $\mathcal S$ is the Schwartz space, and $\|\cdot\|$ is the Euclidean norm on $\mathbb{R}^n$.
\end{df}

\begin{df} Suppose that ${\phi}:\mathbb{R}^n\to\mathbb{C}$ is continuous and slowly increasing.  A measurable function $\hat{{\phi}}\in L^{2}_{loc}(\mathbb{R}^n \setminus \{0\})$ is called the \emph{specialized Fourier transform} of ${\phi}$ if there exists an integer $m\in\mathbb{Z}^+$ such that
\[
\int_{\mathbb{R}^n}{\phi}(x)\hat{\gamma}(x)dx = \int_{\mathbb{R}^n}\hat{\phi}(\xi)\gamma(\xi)d\xi
\]
is satisfied for all $\gamma\in\mathcal{S}_{2m}$.  The least such $m$ is called the order of $\hat{\phi}$.
\end{df}

Before we continue, a few comments are needed.  In \cite{wendland} this is called the generalized Fourier transform.  However, there is already a customary definition of the generalized Fourier transform, which uses the entire Schwartz class as test functions.  To avoid this confusion, I have changed the name.  If a function $\phi(x)\in L^1(\mathbb{R}^n)$ then its specialized and classical Fourier transform coincide, and the order is 0.  The same is true for $\phi(x)\in L^2(\mathbb{R}^n)$.  Finally, the specialized and the generalized (defined in the usual way) Fourier transform coincide on the set $\mathcal{S}_{2m}$.

We turn to constraints on the data that we will use.  It is not necessary to start out assuming that the data comes from any particular function, however it serves our purposes later.  Thus throughout the paper we will sample functions exclusively from the Paley-Wiener space $PW_\pi$.
\begin{df}  We define the \emph{Paley Wiener space}, denoted $PW_\pi$, as 
\begin{equation}\label{PW definition}
PW_\pi = \{f\in L^2(\mathbb{R}^n) : \text{ supp}(\hat{f})\subseteq [-\pi,\pi]^n \}.
\end{equation}
We note that the norm on this space is given by
\[
\| f \|_{PW_\pi}=\| f \|_{L^2(\mathbb{R}^n)}.
\]
\end{df}
\begin{lem}\label{sq sum}
Suppose that $f\in PW_\pi$, then $\{f(j)\}\in l^2(\mathbb{Z}^n)$.
\end{lem}

\begin{proof}
This is a direct consequence of Plancherel's theorem.
\end{proof}

\section{Cardinal Interpolators}
Here we set out hypotheses on a function $\phi(x)$ which will be used throughout the rest of the paper.
\begin{df}  We say that a function $\phi(x)$ is a \emph{cardinal interpolator} if it satisfies the following conditions:
\begin{enumerate}
\item[(H1)]$\phi(x)$ is a real valued slowly increasing function on $\mathbb{R}^n$,
\item[(H2)] $\hat{{\phi}}(\xi) \geq 0$ and $\hat{{\phi}}(\xi)\geq \delta >0$ in $[-\pi,\pi]^n$,
\item[(H3)] $\hat{{\phi}}(\xi)\in C^{n+1}(\mathbb{R}^n\setminus\{0\})$,
\item[(H4)] There exists $\epsilon >0$ such that if $|\alpha|\leq n+1$, \\ $D^{\alpha}\hat\phi(\xi)=O(\| \xi \|^{-(n+\epsilon)})$ as $\|\xi\|\to\infty$,
\item[(H5)] For any multi-index $\alpha$, with $|\alpha|\leq n+1$, 
\[
\dfrac{\displaystyle\prod_{j=1}^{|\alpha|}D^{\alpha_j}\hat\phi(\xi)  }     {[\hat\phi(\xi)]^{|\alpha|+1}} \in L^{\infty}([-\pi,\pi]^n) \hspace{.25 in}\text{where}\hspace{.25in} \sum_{j=1}^{|\alpha|}\alpha_j = \alpha.
\] 
\end{enumerate}

\end{df}

For a cardinal interpolant $\phi(x)$, we form the corresponding so-called \emph{fundamental function} of interpolation, denoted $L_\phi(x)$, by its Fourier transform.
\begin{equation}\label{fundamental function}
\hat{L}_\phi(\xi)=(2\pi)^{-n/2}\hat{\phi}(\xi)\left[\sum_{j\in\mathbb{Z}^n}\hat{\phi}(\xi-2\pi j)\right]^{-1}
\end{equation}

\section{Properties of the Fundamental Function}
The main goal of this section is to establish the integrability $L_\phi(x)$.  This is done through several lemmas.  First we rewrite the transform of the fundamental function as
\[
(2\pi)^{n/2}\hat{L}_\phi(\xi) = \dfrac{\hat{\phi}(\xi)}{\hat{\phi}(x)+u(\xi)},
\]
where $u(\xi)=\sum_{j\neq 0}\hat{\phi}(\xi-2\pi j)$.  Condition (H4) implies that for all $|\alpha|\leq n+1$, $D^\alpha u$ is well defined on $[-\pi,\pi]^n$.  Notice also that (H2) provides the transparent bound $|\hat{L}_\phi(\xi)| \leq 1$.  Hypotheses (H4) and (H5) will allow us to control the growth of the derivative of the above expression.  Before stating our first lemma we adopt the notation $\mathbb{N}_0= \mathbb{N}\cup\{0\}$, and follow the convention $D^0 f = f$. 

\begin{lem} If $\alpha\in \mathbb{N}_0$ is a multi-index, with $|\alpha|\geq 1$, we have
\begin{equation}\label{L deriv}
D^{\alpha}\hat{L}_\phi (\xi) = \left[ \hat{\phi}(\xi)+u{\xi} \right]^{-|\alpha|-1}\sum_{a\in A_\alpha,B\subseteq S_\alpha} C_{a,B}\prod_{\gamma\in B}D^{a_\gamma}\hat{\phi}(\xi)\prod_{\delta\notin B}D^{a_\delta}u(\xi),
\end{equation}
where $A_\alpha = \left\{(a_1,a_2,\dots,a_{|\alpha|+1}): a_j\in \mathbb{N}^{n}_{0}, a_1+\cdots+a_{|\alpha|+1}=\alpha  \right\}$, $S_\alpha = \{1,2,\dots,|\alpha|+1 \}$, and $C_{a,B}$ are constants.  Furthermore, the coefficients corresponding to $B=\emptyset$ and $B=S_\alpha$ are $0$.
\end{lem}

\begin{proof}
We will induct on $|\alpha|$.  If $|\alpha|=1$, we use the quotient rule to see that
\[
D^{\alpha}\hat{L}_\phi (\xi)= (2\pi)^{-n/2}\dfrac{u(\xi)D^\alpha\hat\phi(\xi)- \phi(\xi)D^\alpha u(\xi)}{(\hat\phi(\xi)+u(\xi))^2}.
\]
This is the desired form, and since all of the terms on top are mixed products of $\hat\phi(\xi)$ and $u(\xi)$ the coefficient condition holds as well.  Now if we assume that \eqref{L deriv} holds for all $|\alpha|\leq n$, then applying the Leibniz rule and collecting terms yields the result.
\end{proof}

We have the following result.
\begin{cor}\label{cor 1}
If $|\alpha|\leq n+1$, then $D^{\alpha}\hat{L}_{\phi}\in L^{\infty}([-\pi,\pi]^n)$.
\end{cor}
\begin{proof}
Combine (H5) with \eqref{L deriv}.
\end{proof}

In order to determine the growth of $\hat{L}_{\phi}(\xi)$ away from the origin, we examine the function
\[
P(\xi)=\dfrac{1}{\hat\phi(\xi)+u(\xi)}.
\]

\begin{lem}
For $\xi\in[-\pi,\pi]^n$ and $\alpha$ a multi-index, we have
\begin{equation}\label{P deriv}
D^\alpha P (\xi) =\left[  \hat\phi(\xi)+u(\xi) \right]^{-|\alpha|-1} \sum_{a\in A_\alpha}C_a \prod_{j=1}^{|\alpha|}D^{a_j}(\hat\phi(\xi)+u(\xi)),
\end{equation}
where $A_\alpha=\left\{(a_1,a_2,\dots,a_{|\alpha|}): a_j\in \mathbb{N}_0^n, a_1+\cdots+a_{|\alpha|}=\alpha    \right\}$.
\end{lem}
\begin{proof}
We again induct on $|\alpha|$.  If $|\alpha|=1$, we use the quotient rule to obtain
\[
D^\alpha P(\xi) = \dfrac{-D^\alpha (\hat\phi(\xi)+u(\xi))}{(\hat\phi(\xi)+u(\xi))^{2}}.
\]
This is the desired result.  Suppose now that \eqref{P deriv} holds for all $|\alpha|\leq n$, then applying the quotient rule and collecting terms yields the result.
\end{proof}

\begin{cor}\label{cor 2}
For $|\alpha|\leq n+1$, $D^\alpha P(\xi)\in L^{\infty}(\mathbb{R}^n)$.
\end{cor}
\begin{proof}
Since $P(\xi)$ is periodic, we need only combine (H5) and \eqref{P deriv}.
\end{proof}

\begin{prop}\label{prop 1}
If $\phi$ is a cardinal interpolator and $|\alpha|\leq n+1$, then $D^\alpha \hat{L}_\phi\in L^1(\mathbb{R}^n)$.
\end{prop}

\begin{proof}
Since $\hat{L}_\phi(\xi)=\hat\phi(\xi)P(\xi)$, we see that near the origin Corollary \ref{cor 1} controls the growth, while away from the origin Corollary \ref{cor 2} and (H4) imply the result. 
\end{proof}

\begin{cor}\label{cor 3}
If $\phi$ is a cardinal interpolator, then $L_\phi(x)$ is continuous and there exists a constant $C>0$ such that
\begin{equation}\label{L bound}
\left|(1+\| x \|^{n+1})L_\phi(x)   \right| \leq C.
\end{equation}
\end{cor}
\begin{proof}
Continuity follows from Proposition \ref{prop 1} by letting $|\alpha|=0$.  The bound \eqref{L bound}, follows as well since
\[
\left|(1+\| x \|^{n+1})L_\phi(x)   \right|\leq \sum_{|\alpha|\leq n+1}\left\| D^\alpha\hat{L}_\phi (\xi) \right\|_{L^1(\mathbb{R}^n)}.
\]
\end{proof}
We now show why $L_{\phi}(x)$ is called a fundamental function of interpolation.
\begin{lem}
$L_{{\phi}}(x)$ satisfies $L_{{\phi}}(k)=\delta_{0,k}$ for $k\in\mathbb{Z}^n$.
\end{lem}
\begin{proof}
The proof is an easy calculation invoking the inversion theorem:
\begin{align*}
L_{{\phi}}(k)=&\dfrac{1}{(2\pi)^{n/2}}\int_{\mathbb{R}^n}\hat{L}_{{\phi}}(\xi)e^{ik\cdot\xi}d\xi \\
=&\dfrac{1}{(2\pi)^{n/2}}\sum_{m\in\mathbb{Z}^n}\int_{[-\pi,\pi]^n}\hat{L}_{{\phi}}(\xi-2\pi m)e^{ik\cdot\xi}d\xi \\
=&\dfrac{1}{(2\pi)^{n/2}}\int_{[-\pi,\pi]^n}e^{ik\cdot\xi}\sum_{m\in\mathbb{Z}^n}\hat{L}_{{\phi}}(\xi-2\pi m)d\xi \\
=&\dfrac{1}{(2\pi)^n}\int_{[-\pi,\pi]^n}e^{ik\cdot\xi}d\xi =\delta_{0,k}.
\end{align*}
\end{proof}

We turn now to constructing an interpolant consisting of translates of $L_{\phi}(x)$.
\begin{claim}If $\{f(j)\}_{j\in\mathbb{Z}^n}\in l^{\infty}(\mathbb{Z}^n)$ and we define $I_{{\phi}}f(x)$ pointwise as follows:
\begin{equation}\label{I def}
I_{{\phi}}f(x)=\sum_{j\in\mathbb{Z}^n}f(j)L_{{\phi}}(x-j).
\end{equation}
where $\phi(\xi)$ is a cardinal interpolator, then $I_{{\phi}}f(x)$ is well defined and continuous.
\end{claim}
\begin{proof}
As a result of \eqref{L bound} we have the following:
\begin{equation}\label{claimbd}
|I_{{\phi}}f(x)|\leq C\|\{f(j)  \}  \|_{l^{\infty}(\mathbb{Z}^n)}\sum_{j\in\mathbb{Z}^n}\dfrac{1}{1+\|x-j\|^{n+1}} < \infty.
\end{equation}
Continuity follows from the Weierstrass M test and the continuity of $L_{\phi} (x)$.  Notice this bound holds for all $x\in\mathbb{R}^n$ because the last sum is periodic.
\end{proof}

We are now in position to prove our first theorem.
\begin{thm}
Let $1\leq p \leq \infty$.  If $\{f(j)\}\in l^p( \mathbb{Z}^n)$ and $I_{{\phi}}f(x)$ is defined by \eqref{I def}, then $I_{{\phi}}f(x)\in L^p(\mathbb{R}^n)$.
\end{thm}
\begin{proof}
We will use the Riesz-Thorin interpolation theorem.  Consider the linear map given by $\{f(j)\}\mapsto I_{\phi} f(x)$.  From the above claim, we can see that $I_\phi$ maps $l^\infty(\mathbb{Z}^n)$ to $C(\mathbb{R}^n)$.  In order to apply the theorem, we must check the $p=1$ case and the $p=\infty$ case.  For $p=1$ we have:
\begin{align*}
\|  I_{\phi} f(x)  \|_{L^1(\mathbb{R}^n)} =& \left\| \sum_{j\in\mathbb{Z}^n}f(j)L_{\phi}(x-j)  \right\|_{L^1(\mathbb{R}^n)} \\
\leq& \| \{f(j)\} \|_{l^1(\mathbb{Z}^n)} \| L_{{\phi}}(x-j) \|_{L^1(\mathbb{R}^n)}\\
\leq& C\bigg\|\dfrac{1}{1+\|x\|^{n+1}}\bigg\|_{L^1(\mathbb{R}^n)}  \| \{f(j)\} \|_{l^1(\mathbb{Z})} \\
\leq& C_1\|\{ f(j) \}  \|_{l^1(\mathbb{Z}^n)}
\end{align*}

Here we've used Minkowski's inequality and Corollary \ref{cor 3}. 
The $p=\infty$ case is the content of \eqref{claimbd}.  Thus the Riesz-Thorin interpolation theorem guarantees a constant $C_p$ independent of $\{f(j)\}$ such that
\[
\|  I_{\phi} f(x)  \|_{L^p(\mathbb{R}^n)}\leq C_p\|\{f(j)\}  \|_{l^{p}(\mathbb{Z}^n)}.
\]
\end{proof}

\section{Regular Families of Cardinal Interpolators}
The goal of this section is to introduce regularity conditions on a family of cardinal interpolators $\{\phi_\alpha\}$ so that we may get recovery results similar to those found in \cite{baxter}, \cite{splines}, \cite{RS 1}, \cite{RS 2}, and \cite{RS}.

\begin{df}
We call a family of functions $\{\phi_{\alpha}(\xi)\}_{\alpha\in A}$ a \emph{regular family of cardinal interpolators} if for each $\alpha\in A\subset(0,\infty)$, $\phi_{\alpha}(x)$ is a cardinal interpolator and in addition to this, we have:
\begin{enumerate}
\item[(R1)] For $j\in\mathbb{Z}^n\setminus\{0\}$, $\xi\in[-\pi,\pi]^n$, define $M_{j,\alpha}(\xi)=\displaystyle\dfrac{\hat{\phi}_\alpha(\xi+2\pi j)}{\hat{\phi}_\alpha(\xi)} $, then for $j\in\mathbb{Z}^n\setminus\{0\}$, $\displaystyle\lim_{\alpha\to\infty}M_{j,\alpha}(\xi)=0$ for almost every $\xi\in[-\pi,\pi]^n$.
\item[(R2)] There exists $M_j\in l^1(\mathbb{Z}^n\setminus\{0\})$, independent of $\alpha$, such that for all $j\in\mathbb{Z}^n\setminus\{0\}$, $M_{j,\alpha}(\xi)\leq M_j$.
\end{enumerate}
\end{df}

Depending on the interpolator $\phi(x)$, the parameter $\alpha$ may be continuous or discrete.  We will encounter examples of both in the following section.

\begin{prop}\label{L limit} If $\{\phi_\alpha\}$ is a regular regular family of cardinal interpolators, then $\hat{L}_{{\phi}_\alpha}(\xi) \to (2\pi)^{-n/2}\chi_{[-\pi,\pi]^n}(\xi)$ almost everywhere.
\end{prop}
\begin{proof}
For $\xi\in[-\pi,\pi]^n$ we have:
\[
\hat{L}_{{\phi}_\alpha}(\xi)=(2\pi)^{-n/2}\bigg\{1+\dfrac{u_{\alpha}(\xi)}{\hat{\phi}_{\alpha}(\xi)}   \bigg\}^{-1}.
\]
If we can show $\dfrac{u_{\alpha}(\xi)}{\hat{\phi}_{\alpha}(\xi)}  \to 0$ we will be done with this part.  Conditions (R1) and (R2) allow us to apply the dominated convergence theorem for series, which gives us the result.  Now we have only to check $\xi\notin [-\pi,\pi]^n$.  We may write $\xi=2\pi j + r$, where $j\in\mathbb{Z}^n\setminus\{0\}$ and $r\in[-\pi,\pi]^n$.  Then we have:
\[
\hat{L}_{{\phi}_\alpha}(\xi)=(2\pi)^{-n/2} \dfrac{\hat{\phi}_{\alpha}(2\pi j + r)}{\hat{\phi}_{\alpha}(r)}\bigg\{1+\dfrac{u_{\alpha}(r)}{\hat{\phi}_{\alpha}(r)}   \bigg\}^{-1} \leq (2\pi)^{-n/2} M_{j,\alpha}(\xi) .
\]
The last inequality comes from (H2), thus (R1) provides the desired result.
\end{proof}

\begin{prop}\label{u lim}
If $\{\phi_\alpha\}$ is regular a regular family of cardinal interpolators, then $\displaystyle \lim_{\alpha\to\infty}\sum_{j\neq0}\hat{L}_{\phi_\alpha}(\xi+2\pi j)=0$ almost everywhere and in $L^1([-\pi,\pi]^n)$.
\end{prop}
\begin{proof}
We note that if $j\in\mathbb{Z}^n\setminus\{0\}$, then $\hat{L}_{\phi_\alpha}(\xi+2\pi j) \leq M_{j,\alpha}(\xi)$ by (R1), thus (R2) allows us to use the dominated convergence theorem for series which gives the result.
\end{proof}

These propositions allow us to prove the main theorem.
\begin{thm}
Suppose that $\{\hat{{\phi}}_{\alpha}(\xi)\}_{\alpha\in A}$ is a regular family of cardinal interpolators and let $f(x)\in PW_\pi$.  Then we have the following limits:
\begin{enumerate}
\item[$(a)$] $\displaystyle \lim_{\alpha\to\infty}\bigg\| \sum_{j\in\mathbb{Z}^n}f(j)L_{{\phi}_\alpha}(x-j) - f(x) \bigg \|_{L^2(\mathbb{R}^n)} =0,$
\item[$(b)$] $\displaystyle \lim_{\alpha\to\infty}\bigg| \sum_{j\in\mathbb{Z}^n}f(j)L_{{\phi}_\alpha}(x-j) - f(x) \bigg| =0$  uniformly in $\mathbb{R}^n$.
\end{enumerate}
\end{thm}

\begin{proof}
$(a)$  Since $f\in PW_\pi$, we have that $\{f(j)\}_{j\in\mathbb{Z}^n}\in l^2(\mathbb{Z}^n)$.  This, in turn, implies that $\displaystyle \sum_{j\in\mathbb{Z}^n}f(j)L_{{\phi}_\alpha}(x-j) \in L^2(\mathbb{R}^n)$.  So we may use Plancherel's theorem to estimate the norm.
\begin{align*}
 &\left\| \sum_{j\in\mathbb{Z}^n}f(j)L_{{\phi}_\alpha}(x-j) - f(x) \right\|^{2}_{L^2(\mathbb{R}^n)} \\
=&\left\| \hat{L}_{{\phi}_\alpha}(\xi)\sum_{j\in\mathbb{Z}^n}f(j)e^{-ij\cdot\xi} -\hat{f}(\xi) \right\|^{2}_{L^2(\mathbb{R}^n)} \\
=&\int_{[-\pi,\pi]^n}\left|\hat{f}(\xi)\left((2\pi)^{n/2}\hat{L}_{{\phi}_\alpha}(\xi)-1\right)\right|^2d\xi \\
+&\int_{\mathbb{R}^n\setminus[-\pi,\pi]^n} \left| \hat{L}_{{\phi}_\alpha}(\xi)\sum_{j\in\mathbb{Z}^n}f(j)e^{-ij\cdot\xi}\right|^2d\xi \\
\end{align*}
We have used that $f\in\ PW_\pi$ and the well known $L^2(\mathbb{R}^n)$ equality $\displaystyle (2\pi)^{n/2}\hat{f}(\xi)=\sum_{j\in\mathbb{Z}^n}f(j)e^{-ij\cdot\xi}$ for $\xi\in[-\pi,\pi]^n$.
Proposition \ref{L limit} shows that  $(2\pi)^{n/2}\hat{L}_{{\phi}_\alpha}(\xi) \to \chi_{[-\pi,\pi]^n}(\xi)$ almost everywhere,  thus the first integral tends to 0 by the dominated convergence theorem, since we have the elementary estimate $|(2\pi)^{n/2}\hat{L}_{{\phi}_\alpha}(\xi)-1|\leq (2\pi)^{n/2}+1$.  We estimate the second integral.

\begin{align*}
&\int_{\mathbb{R}^n\setminus[-\pi,\pi]^n} | \hat{L}_{{\phi}_\alpha}(\xi)\sum_{j\in\mathbb{Z}^n}f(j)e^{-ij\cdot\xi}|^2d\xi \\
=& \sum_{m\in\mathbb{Z}^n\setminus\{0\}}\int_{[-\pi,\pi]^n}|\hat{L}_{{\phi}_\alpha}(\xi+ 2\pi m)\sum_{j\in\mathbb{Z}^n}f(j)e^{-ij\cdot(\xi+2\pi m)}|^2d\xi \\
=&(2\pi)^{n}\sum_{j\in\mathbb{Z}^n\setminus\{0\}}\int_{[-\pi,\pi]^n}|\hat{L}_{{\phi}_\alpha}(\xi+ 2\pi j)|^2|\hat{f}(\xi)|^2d\xi  \\
\leq& (2\pi)^n \int_{[-\pi,\pi]^n}\sum_{j\in\mathbb{Z}^n\setminus\{0\}}|M_{j,\alpha}(\xi)|^2|\hat{f}(\xi)|^2d\xi
\end{align*}
This last term tends to 0 by the dominated convergence theorem as in Proposition \ref{u lim}, noting that $l^1(\mathbb{Z}^n\setminus\{0\})\subset l^2(\mathbb{Z}^n\setminus\{0\})$.

Part $(b)$ is proved in a similar manner, with help from the Cauchy-Schwarz inequality.
\begin{align*}
&\left| \sum_{j\in\mathbb{Z}^n}f(j)L_{{\phi}_\alpha}(x-j) - f(x)  \right| \\
=&(2\pi)^{-(n/2)}\left|   \int_{\mathbb{R}^n} \left(\hat{L}_{{\phi}_\alpha}(\xi)\sum_{j\in\mathbb{Z}^n}f(j)e^{-ij\xi}  -\hat{f}(\xi) \right)e^{ix\cdot\xi} d\xi \right| \\
\leq& \int_{[-\pi,\pi]^n}\left|\hat{L}_{{\phi}_\alpha}(\xi)-(2\pi)^{-n/2}\right|\left|\hat{f}(\xi)\right|d\xi \\
+&\sum_{j\neq 0}\int_{[-\pi,\pi]^n}\left|\hat{L}_{{\phi}_\alpha}(\xi+2\pi j)\right|\left|\hat{f}(\xi)\right|d\xi  \\
\leq& \|f\|_{PW_\pi} \left\|\hat{L}_{{\phi}_\alpha}(\xi)-(2\pi)^{-n/2}\right\|_{L^2([-\pi,\pi]^n)} \\
+&\|f\|_{PW_\pi}\left\|\sum_{j\neq 0}M_{j,\alpha}(\xi) \right\|_{L^2([-\pi,\pi]^n)}
\end{align*}
From Proposition \ref{L limit} above we can see that the first term tends to 0.  By (R2) we have that the sum in the second term is bounded above by $\|M_j\|_{l^1(\mathbb{Z}^n\setminus\{0\})}^2$, so Proposition \ref{u lim} and the dominated convergence theorem finish the proof. 
\end{proof}

\section{Examples of Regular Families of Interpolators}
In this section, we exhibit three families of regular cardinal interpolators.  Our first example will be that of polyharmonic cardinal splines.  The result that we attain is most similar to the one found in \cite{splines}.  Our second example will be a family of Gaussians, which has also been studied previously, see for instance \cite{RS 1}, \cite{RS 2}.  Finally, we will consider a family of multiquadrics and extend a result which is found in \cite{RS}, as well a provide what appears to be a novel example.  For each example we work out the one dimensional case, followed by the multivariate case often suppressing the straightforward yet tedious details involved in checking (H5).

\subsection{Polyharmonic Cardinal Splines}
Our first example was explored in \cite{splines}, and motivated this treatment.  We begin by considering the one dimensional version of the example found there.  A function or distribution $f$ is called $k$-$harmonic$, where $k\in\{1,2,3,\dots\}$, if it satisfies
\begin{equation}\label{kharm}
\Delta^k f =0
\end{equation}    
on $\mathbb{R}^n$, where $\Delta$ is the Laplacian operator and  $\Delta^k( f)=\Delta(\Delta^{k-1} f)$.  A function which satisfies \eqref{kharm} with $k\geq 1$ is called $polyharmonic$.  The fundamental solution of $\eqref{kharm}$ is given by a constant multiple of $\phi_k(x)=|x|^{2k-1}$, whose Fourier transform is given by $\hat{\phi}_k(\xi)=(2\pi)^{-1/2}|\xi|^{-2k}$.

A polyharmonic cardinal spline is a function or distribution on $\mathbb{R}$ which satisfies
\begin{enumerate}
\item[(i)] $f\in C^{2k-2}(\mathbb{R}^n)$,
\item[(ii)] $\Delta^k f =0 \hspace{.5in}\text{on }\mathbb{R}^n\setminus\mathbb{Z}^n$.
\end{enumerate}
The fundamental functions $L_k(x)$ use $\hat{\phi}_k(\xi)=(2\pi)^{-1/2}|\xi|^{-2k}$.  Several properties of the fundamental function may be found in \cite{splines}, including the fact that it is a $k$-harmonic spline.  We show that $\{|\xi|^{-2k}\}_{k=1}^{\infty}$ is a regular family of cardinal interpolators.  Checking (H1)-(H5) is an elementary exercise in calculus and will be omitted.  We check (R1) and (R2).  For $j\neq0$ and $|\xi|<\pi$ we have
\[
M_{j,k}(\xi)=\dfrac{|\xi|^{2k}}{|\xi+2\pi j|^{2k}}\leq (2|j|-1)^{-2k}.
\]
These terms tend to 0, hence (R1) is satisfied.  For (R2) we note that
\[
 (2|j|-1)^{-2k} \leq  (2|j|-1)^{-2}.
\]
Thus, if $n=1$, polyharminic cardinal splines form a regular family of cardinal interpolators.
For $n>1$, we have $\hat{\phi}_{k}(\xi)=\|\xi\|^{-2k}$, where $\xi\in\mathbb{R}^n$, in this case we must take $k\geq n+1$.  (H1)-(H4) are straightforward, so we will check (H5), (R1), and (R2).  We need to find the derivatives in question.  To this end, we note that for $j=1,\dots,n$ we have
\[
\dfrac{d}{d\xi_j}\left(\|\xi\|^{-2k}\right)=-2k\|\xi\|^{-2k-1}\dfrac{\xi_j}{\|\xi\|}.
\] 
Repeating this argument for higher derivatives we arrive at
\[
D^{\beta}\left(\|\xi\|^{-2k}\right)=\|\xi\|^{-2k-|\beta|}\Omega(\xi)
\]
where $\beta$ is a multi-index such that $|\beta|\leq n+1$, and $\Omega(\lambda\xi)=\Omega(\xi)$.    
This information allows us to check (H5) easily.  We move on to (R1) and (R2).  Mimicking the one dimension case we see that 
\[
M_{j,k}(\xi)=\left( \dfrac{\|\xi\|}{\|\xi+2\pi j   \|}  \right)^{2k}.  
\]
If $\xi\in(-\pi,\pi)^n$ and $j\neq0$, then we have $\|\xi\|<\|\xi+2\pi j\|$ hence these terms go to 0.  Thus we have shown that (R1) holds.  For (R2) we must be a little more careful.  In this case, we see that the terms grow like $\|j\|^{-2k}$.  Thus we may find a constant and C such that $M_{j,k}(\xi)\leq C \| j\|^{-2k}$ as long as $j\neq0$.  Recalling that $2k\geq n+1$ we have that
\[
M_{j,k}(\xi)\leq C\|j\|^{-n-1}.
\]  
 
\subsection{Gaussians}   
Our second example is the Gaussian.  This example was studied in \cite{RS 1, RS 2}, although we change the notation somewhat.  We begin with $\phi_{\alpha}(x)=e^{-x^2/(4\alpha)}$, where $x\in\mathbb{R}$, thus omitting constants we have $\hat{{\phi}}_\alpha(\xi)= e^{-\alpha\xi^2}$.  We take $\alpha\geq1$ as a convenience, the arguments go through with $\alpha\geq\alpha_0>0$.  Again, checking conditions (H1)-(H5), is a straightforward calculus exercise, whose details are omitted.  We check (R1) and (R2).  For $|\xi|<\pi$ and $j\neq0$, we have
\[
M_{j,\alpha}(\xi)= e^{\alpha\xi^2-\alpha(\xi+2\pi j)^2}.
\]
We can see that these terms go to 0 since the exponent is negative.  Since this is true for all $j\neq0$, (R1) holds. As for (R2) we note that if $\alpha\geq1$ we have
\[
e^{\alpha\pi^2-\alpha(\xi+2\pi j)^2}=e^{-\alpha\pi^2(4|j|)(|j|-1)}\leq e^{-4\pi^2|j|(|j|-1)}.
\]
Hence $\left\{e^{-x^2/(4\alpha)}\right\}_{\alpha\geq1}$ is a regular family of cardinal interpolators in the univariate case.  In the multivariate case, we have $\hat{{\phi}}_\alpha(\xi)= e^{-\alpha\|\xi\|^2}$, and again consider $\alpha\in[1,\infty)$  This function is smooth so all of the relevant properties are easily checked.  (D1) and (D2) are checked as above by replacing the absolute values with norms.  Hence, $\left\{e^{-x^2/(4\alpha)}\right\}_{\alpha\geq1}$ is a regular family of cardinal interpolators in the multivariate case as well.

\subsection{Multiquadrics}

Our next example is a family of `multiquadrics', meaning $\phi(x)=(c^2+\|x\|^2)^\alpha$.  In this example we have a choice between which parameter to limit over and which to keep fixed.  We will work both cases.  In the univariate case, allowing $c$ to vary is the subject of \cite{RS}.  We begin by finding the specialized Fourier transform of $\phi(x)=(c^2+\|x\|^2)^{\alpha}$.  We have
\[
\hat{\phi}(\xi)=\dfrac{2^{1+\alpha}}{\Gamma(-\alpha)}\left( \dfrac{\|\xi\|}{c}  \right)^{-\alpha-\frac{n}{2}}K_{\frac{n}{2}+\alpha}(c\|\xi\|), \hspace{.5in}\xi\neq0,
\]
where
\[
K_\beta(r) = \dfrac{1}{2}\int_{\mathbb{R}}e^{-r\cosh t}e^{\beta t}dt,
\]
for $r>0$ and $\beta\in\mathbb{R}$.  The details of this may be found in Theorem 8.15 of \cite{wendland}.  We must omit the case when $\alpha\in\{0,1,2,\dots\}$ because in this case we would get a measure in the transform space and we would like a function.  This is why we use the specialized Fourier transform and not the more standard generalization. 
We first consider the univariate case and list some properties of $K_\beta(r)$ which will be useful is our analysis. These properties may be found on page 52 of \cite{wendland}.
The first is that $K_\beta(r)$ is decreasing on $(0,\infty)$.  We find the following asymptotic expansions in \cite{AS}:
\begin{align}
K_\beta(r)&\sim \dfrac{1}{\sqrt{r}}e^{-r}e^{|\beta|^2/(2r)}\hspace{.25in}&(r\to\infty),\\
\label{0bnd}K_\beta(r)&\sim r^{-|\beta|}&(r\to0).
\end{align}
Here $f\sim g$ means that $f=(C+o(1))g$, where $C$ is a constant independent of the limiting parameter.  We also have the differentiation formula
\begin{equation}\label{Kderiv}
\dfrac{d}{dz}\big[ z^\beta K_{\beta}(z)   \big] = -z^{\beta}K_{\beta -1}(z).
\end{equation}
And because we will need it later we have the following bounds for $|\beta|\geq\frac{1}{2}$ and $r>0$
\begin{equation}\label{doublebnd}
\sqrt{\dfrac{\pi}{2}}r^{-1/2}e^{-r}\leq K_\beta(r)\leq \sqrt{2\pi}r^{-1/2}e^{-r}e^{-\beta^2/(2r)}
\end{equation}
All of this information may be found in \cite{AS} and Chapter 5 of \cite{wendland}, and has been restated here for convenience and to aide in the calculations that follow.  We will consider two different families of multiquadrics.  In the first example we fix a positive number $c$ and take $\{\alpha_j\}_{j=1}^{\infty}\subset[\frac{1}{2},\infty)$, such that dist$\left(\{\alpha_j\},\{1,2,3,\dots\}\right)>0$ and $\displaystyle\lim_{j\to\infty}\alpha_j=\infty$.  We show $\{\phi_{\alpha_j}(x)\}_{j=1}^{\infty}$ is a regular family of cardinal interpolators.  We begin in the univariate case checking (H1)-(H5).  (H1)-(H4) follow from the definition, the differentiation formula, and the asymptotic estimates.  (H5) also follows from the asymptotic estimates, but we will show these calculations.  First, we find the derivatives in question omitting the associated constants since they have no bearing on the overall result.
\begin{align*}
\hat{\phi}^{(1)}_{\alpha_j}(\xi)=&-|\xi|^{-\alpha_j-\frac{1}{2}}K_{\alpha_j+\frac{3}{2}}(c|\xi|)\\
\hat{\phi}^{(2)}_{\alpha_j}(\xi)=&|\xi|^{-\alpha_j-\frac{1}{2}}K_{\alpha_j-\frac{5}{2}}(c|\xi|)+|\xi|^{-\alpha_j-\frac{3}{2}}K_{\alpha_j+\frac{3}{2}}(c|\xi|)
\end{align*}
Using the asymptotic expansion, we find the following limits.
\begin{align*}
\dfrac{1}{\hat{\phi}_{\alpha_j}(\xi)}&=\dfrac{|\xi|^{\alpha_j+\frac{1}{2}}}{K_{\alpha_j+\frac{1}{2}}}=\dfrac{|\xi|^{2\alpha_j+1}}{C+o(1)} = 0  &\hspace{.25in}(|\xi|\to0)\\
\dfrac{\hat{\phi}^{(1)}_{\alpha_j}(\xi)}{\hat{\phi}^{2}_{\alpha_j}(\xi)}& = |\xi|^{\alpha_j+\frac{1}{2}}\dfrac{K_{\alpha_j+\frac{3}{2}}(c|\xi|)}{K_{\alpha_j+\frac{1}{2}}^{2}(c|\xi|)}=|\xi|^{2\alpha_j}\dfrac{C_2+o(1)}{[C_1+o(1)]^2}=0 &(|\xi|\to0)\\
\dfrac{\hat{\phi}^{(2)}_{\alpha_j}(\xi)}{\hat{\phi}^{2}_{\alpha_j}(\xi)}&=|\xi|^{2\alpha_j-1}\left\{\dfrac{C_3+o(1)}{[C_1+o(1)]^2} -\dfrac{C_2+o(1)}{[C_1+o(1)]^2}   \right\} &(|\xi|\to0)
\end{align*}
We can see that in the last expression we must take $\alpha_j\geq\dfrac{1}{2}$ in order for the limit to be bounded.  Now we turn to (R1).  Let $|\xi|<\pi$ and $i\neq0$, then
\[
M_{i,\alpha_j}(\xi)\leq\dfrac{\pi^{\alpha_j+\frac{1}{2}}}{K_{\alpha_{j}+\frac{1}{2}}(c\pi)}\dfrac{K_{\alpha_j+\frac{1}{2}}(c|\xi+2\pi i|)}{|\xi+2\pi i|^{\alpha_j+\frac{1}{2}}}\leq \dfrac{1}{(2|i|-1)^{\alpha_j+\frac{1}{2}}}.
\]
We have used the fact that $K_\beta(r)$ is decreasing.  This estimate allows us to prove (D2) as well, since the condition for every $\alpha$ may be replaced with for every $\alpha$ large enough.  In this case, large enough means $\alpha_j>\frac{1}{2}$.  Thus we have that $\left\{(c^2+|x|^2)^{\alpha_j}\right\}$ is a well behaved family for fixed $c>0$ and$\{\alpha_j\}_{j=1}^{\infty}\subset[\frac{1}{2},\infty)$, such that dist$\left(\{\alpha_j\},\{1,2,3,\dots\}\right)>0$ and $\displaystyle\lim_{j\to\infty}\alpha_j=\infty$.  Before moving on, we make note of the specific example $\alpha_k=\frac{2k-1}{2}$ for $k\in\{1,2,3,\dots\}$.  This is the family of odd powers of the multiquadric $\{\left(\sqrt{c^2+|x|^2}\right)^{2k-1}\}$.  This family corresponds to `smoothed out' polyharmonic splines of order $2k$.  In light of the results in \cite{splines}, one may well expect that our family of multiquadrics shares a similar convergence property.  This is indeed the case as our work above shows.

We now explore what happens if we fix $\alpha$ and allow $c$ to vary.  Again we will need to bound the parameter away from the origin so we take $c\geq1$.  Let $\phi_c(x)=(c^2+|x|^2)^{\alpha}$ for some fixed $\alpha\in(-\infty,-\frac{3}{2}]\cup[\frac{1}{2},\infty)\setminus\{1,2,3,\dots\}$, then, omitting constants because they will divide out in the fundamental function, we have $\hat{\phi}_{c}(\xi)=|\xi|^{-\alpha-\frac{1}{2}}K_{\frac{1}{2}+\alpha}(c|\xi|)$.  The bounds on $\alpha$ are a result of the restrictions imposed by the derivatives.  Since $\hat{\phi}_c(0)$ is a real positive number if $\alpha<-\frac{1}{2}$, (H5) places restrictions on the growth of $\hat{\phi}_c(\xi)$ near the origin.  To satisfy this condition, we need to increase the rate of decay of $\phi_c(x)$. Most of the work to check that this is a regular family of cardinal interpolators has already been done.  The same arguments used above can be recycled, however we must check (R1) and (R2) separately.  We check (R1) presently.
\begin{align*}
M_{j,c}(\xi)&\leq\dfrac{\pi^{\alpha+\frac{1}{2}}}{|\xi+2\pi j|^{\alpha+\frac{1}{2}}}\dfrac{K_{\alpha+\frac{1}{2}}(c|\xi+2\pi j|)}{K_{\alpha+\frac{1}{2}}(c\pi)} \\
& \leq \left( 2|j|-1 \right)^{-(\alpha+\frac{1}{2})}\dfrac{K_{\alpha+\frac{1}{2}}(c|\xi+2\pi j|)}{K_{\alpha+\frac{1}{2}}(c\pi)}   \\
&\leq  \left( 2|j|-1 \right)^{-(\alpha+1)}e^{c\pi}e^{-c|\xi-2\pi j|}e^{-(\alpha+\frac{1}{2})^2/(2c|\xi+2\pi j|)}\\
&\leq   \left( 2|j|-1 \right)^{-(\alpha+1)}e^{c\pi}e^{-c|\xi-2\pi j|} \\
\end{align*}
We can see, for $j\neq0$ and $|\xi|<\pi$, that $\displaystyle\lim_{c\to\infty}M_{j,c}(\xi)=0$.  This calculation also shows that (R2) is satisfied since each term there is bounded by $(2|j|-1)^{-(\alpha+1)}e^{-2\pi (|j|-1)}$.  Thus we have shown that $\left\{(c^2+|x|^2)^\alpha\right\}_{c\geq1}$ is a well behaved family of interpolators for a fixed $\alpha\in(-\infty,-\frac{3}{2}]\cup[\frac{1}{2},\infty)\setminus\{1,2,3,\dots\}$.  The case where $\alpha=\frac{1}{2}$ is examined in \cite{RS}.  

We now visit the $n$-dimensional multiquadric $\phi(x)=(c^2+\|x\|^2)^\alpha$.  We will again consider both the case when $\alpha$ is the parameter and $c$ is fixed and the case when $c$ is the parameter and $\alpha$ is fixed.  We begin by recalling the specialized $n$-dimensional Fourier transform of the multiquadric,
\[
\hat{\phi}(\xi)=\dfrac{2^{1+\alpha}}{\Gamma(-\alpha)}\left( \dfrac{\|\xi\|}{c}  \right)^{-\alpha-\frac{n}{2}}K_{\frac{n}{2}+\alpha}(c\|\xi\|), \hspace{.5in}\xi\neq0.
\]  
Conditions (H1)-(H4) follow as before.  We need to calculate the derivatives in order to check (H5).  To this end, we may use the differentiation formula \eqref{Kderiv} together with repeated use of the product rule and the asymptotic estimate \eqref{0bnd} to see that near the origin we have the following asymptotic estimate
\begin{equation}\label{Kest}
D^{\beta}\hat\phi(\xi) \sim \|\xi\|^{-\alpha-\frac{n}{2}-|\alpha+\frac{n}{2}+|\beta||}, \hspace{.25in}\|\xi\|\to0.
\end{equation}
Recall that $f\sim g$ means that $f=g(C+o(1))$ for an appropriate constant C.  This calculation holds no matter which parameter we are considering fixed.  Thus we can begin by fixing $c>0$ and checking (H5), (R1) and (R2) for $\phi_\alpha(x)$.  For (H5) we note that $\alpha>0$, hence the asymptotic estimate \eqref{Kest} becomes
\[
D^{\beta}\hat\phi(\xi) \sim \|\xi\|^{-2\alpha-n-|\beta|}, \hspace{.25in}\|\xi\|\to0.
\]
Using this estimate, checking (H5) is now a routine matter.  We note that we must have $\alpha\geq\frac{1}{2}$.  To see (R1), note that
\[
M_{j,\alpha}(\xi)=\left( \dfrac{\|\xi\|}{\|\xi+2\pi j  \|}  \right)^{-\alpha-\frac{n}{2}}\dfrac{K_{\alpha+\frac{n}{2}}(c\|\xi+2\pi j\|)}{K_{\alpha+\frac{n}{2}}(c\|\xi\|)},
\]
and if $\xi\in(-\pi,\pi)^n$ we have
\[
\lim_{\alpha\to\infty}M_{j,\alpha}(\xi)\leq \lim_{\alpha\to\infty}\left( \dfrac{\|\xi\|}{\|\xi+2\pi j  \|}  \right)^{-\alpha-\frac{n}{2}}.
\]
These terms tend to 0 as $\alpha$ increases.  As in the one variable case, we must be careful about how we choose the parameter $\alpha$, in particular we take $\alpha\in\{\alpha_j\}\subset [\frac{1}{2},\infty)$ such that dist$(\{\alpha_j\},\{0,1,2,3,\dots\})>0$ and $\displaystyle\lim_{j\to\infty}\alpha_j=\infty$.  For the bounding sequence in (R2) we take $\alpha=\frac{n+1}{2}$.  These terms grow like $\|j\|^{-n-1}$, hence we may find a constant C such that
\[
M_{j,\alpha}(\xi)\leq C\|j\|^{-n-1}.
\]
This bound holds for all $\alpha_j \geq \frac{n+1}{2}$, hence (R2) holds.

We now will fix $\alpha$ and let $c\to\infty$.  That is, we consider $\phi_c(x)=(c^2+\|x\|^2)^\alpha$.  (H1)-(H4) hold as before.  We take advantage of \eqref{Kest} to find that if $\alpha\in(-\infty,-\frac{2n+1}{2}]\cup[\frac{1}{2},\infty)\setminus\{1,2,3,\dots\}$ then (H5) holds.  As for (R1) and (R2), we must be a bit more careful.  We have
\[
M_{j,c}(\xi)=\left( \dfrac{\|\xi\|}{\|\xi+2\pi j  \|}  \right)^{-\alpha-\frac{n}{2}}\dfrac{K_{\alpha+\frac{n}{2}}(c\|\xi+2\pi j\|)}{K_{\alpha+\frac{n}{2}}(c\|\xi\|)},
\]  
and we can see from \eqref{doublebnd} that
\[
M_{j,c}(\xi)\leq C\left( \dfrac{\|\xi\|}{\|\xi+2\pi j  \|}  \right)^{-\alpha-\frac{n-1}{2}}e^{c(\|\xi\|-\|\xi+2\pi j\|)}.
\]
Here $C$ is a positive constant independent of $c$ whose particular value is unimportant.  If $\xi\in(-\pi,\pi)^n$, then $M_{j,c}(\xi)\to 0$ because the exponent in the exponential is negative.  To check (R2), note that these terms decay exponentially hence we can find a bounding sequence which is in $l^1(\mathbb{Z}^n\setminus\{0\})$.

\end{document}